\documentclass[12pt]{amsart}
\usepackage{geometry}
\usepackage{graphicx}
\usepackage{amssymb}
\usepackage{epstopdf}
\usepackage{hyperref}

\title{Slicing a 2-sphere}
\author{Yevgeny Liokumovich}
\date{}                                           
\begin{document}
\maketitle

\theoremstyle{plain}
\newtheorem{theorem}{Theorem}
\newtheorem{definition}[theorem]{Definition}
\newtheorem{lemma}[theorem]{Lemma}
\newtheorem{proposition}[theorem]{Proposition}
\newtheorem{corollary}[theorem]{Corollary}

\begin{abstract}
We show that for every complete Riemannian surface $M$
diffeomorphic to a sphere with $k \geq 0$ holes
there exists a Morse function $f:M \rightarrow \mathbb{R}$,
which is constant on each connected component of the boundary of $M$
and has fibers of length no more than $52 \sqrt{Area(M)}+length(\partial M)$.
We also show that on every 2-sphere there exists a simple closed curve
of length $\leq 26 \sqrt{Area(S^2)}$ subdividing the sphere into 
two discs of area $\geq \frac{1}{3}Area(S^2)$.
\end{abstract}

\section{Introduction}

Let $M$ be a Reiamannian 2-sphere. Denote the area of $M$ by $|M|$.
In this paper we consider the problem of slicing $M$ by short curves.

We start with the following isoperimetric problem:
when is it possible to subdivide $M$ into two regions of relatively large area
by a short simple closed curve?

Papasoglu \cite{P} used Besicovitch inequality to show that there exists
a simple closed curve of length $\leq 2 \sqrt{3} |M| + \epsilon$
subdividing $M$ into two regions of area $\geq \frac{1}{4} |M|$.
A similar result was independently proved by Balacheff and Sabourau
\cite{BS} using a variation of Gromov's filling argument. 

On the other hand, consider the 3-legged starfish example on Figure \ref{starfish}.

\begin{figure} 
   \centering	
	\includegraphics[scale=0.3]{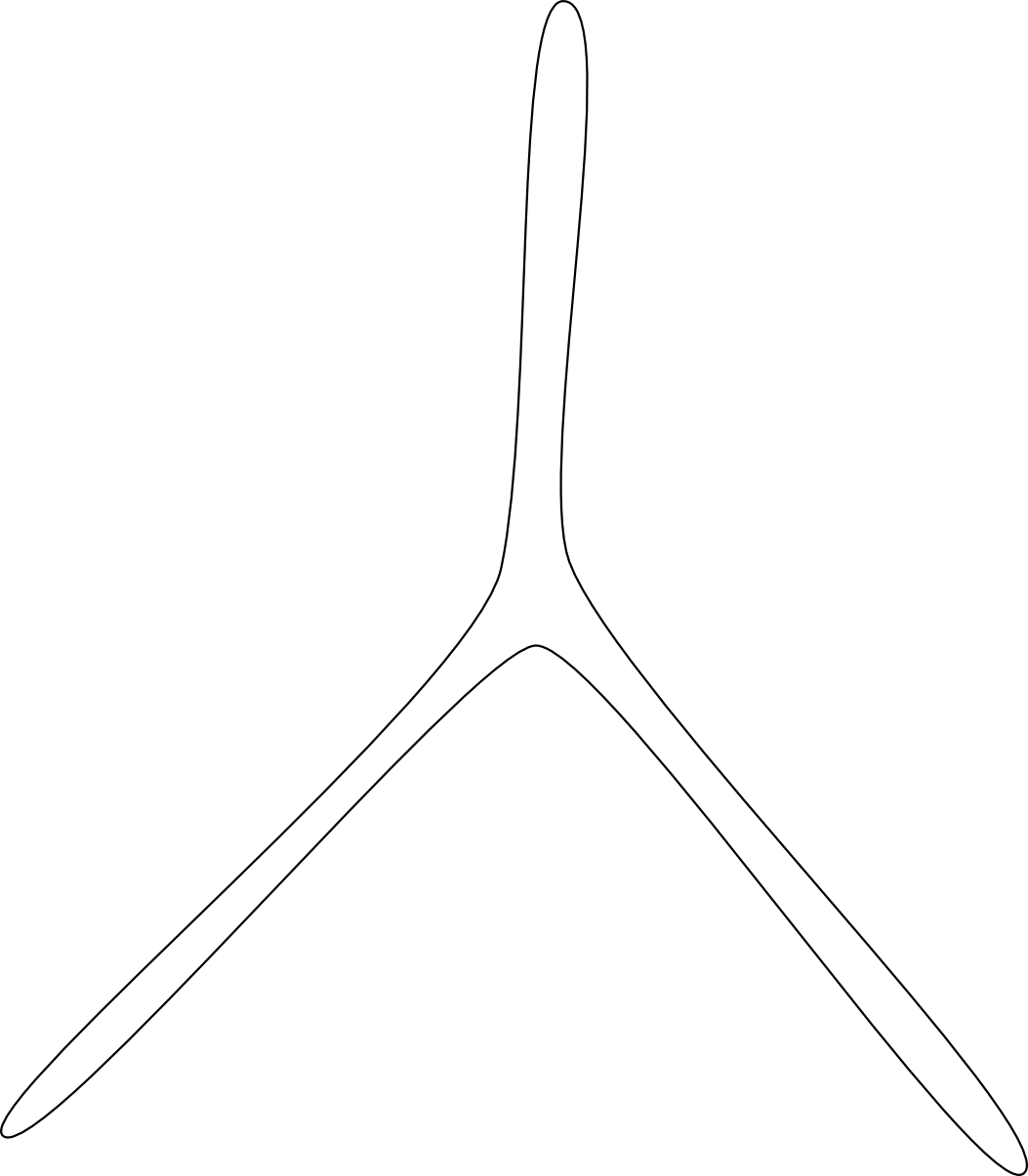}
	\caption{Example of a sphere that can not be subdivided into discs of approximately
	 equal area by a short curve} \label{starfish}
\end{figure}

As it has been observed in \cite{S}, for any $r> \frac{1}{3}$, 
if the tentacles are sufficiently thin and long, the length of
the shortest simple closed curve subdividing $M$ into two regions of area $\geq r|M|$
can be arbitrarily large. A. Nabutovsky asked (\cite{N}) the following question: what is the maximal 
value of $r \in [\frac{1}{4},\frac{1}{3}]$ such that for some
$c(r)$ each Riemannian 2-sphere of area $1$ can be subdivided into
two discs of area $\geq r$ by a simple closed curve of length $\leq c(r)$?

Our first result provides an answer for this question.

\begin{theorem} \label{1/3 area}
There exists a simple closed curve $\gamma$
of length $\leq 26 \sqrt{|M|}$ subdividing $M$ into two
subdiscs of area $\geq \frac{1}{3}|M|$.
\end{theorem}

To prove Theorem \ref{1/3 area} we obtain the following result
of independent interest.

\begin{theorem} \label{tree}
There exists a map $f$ from $M$ into a trivalent tree $T$,
such that fibers of $f$ have length $\leq 26 \sqrt{|M|}$ and controlled 
topology: preimage of every interior point is a simple closed curve,
preimage of every terminal vertex is a point and preimage of every vertex of
degree $3$ is homeomorphic to the greek letter $\theta$.
\end{theorem}

Theorem \ref{tree} follows from a more general Theorem \ref{tree1} in Section 3
for spheres with $k \geq 0$ holes. Using different methods Guth \cite{G} proved existence
of a map from $M$ into a trivalent tree with lengths of fibers bounded in terms of hypersphericity of
$M$. It follows from Theorem 0.3 in \cite{G} that there exists such a map with fibers
of length $\leq 34 \sqrt{|M|}$.

The second main result of this paper is about slicing $M$ by $1-$cycles.
If instead of simple closed curves we allow subdividison by 1-cycles, we show that $M$ can be subdivided 
into two regions with arbitrary prescribed ratio of areas by a 1-cycle of length $\leq 52 \sqrt{|M|}$.
In fact, we prove the following

\begin{theorem} \label{morse}
There exists a Morse function $f:M \rightarrow \mathbb{R}$
with fibers of length $\leq 52 \sqrt{|M|}$.
\end{theorem}

This improves the result of Balacheff and Sabourau \cite{BS}
 that there exists a sweep-out of $M$ by $1-$cycles of length $\leq 10^8 \sqrt{|M|}$.

Theorem \ref{morse} follows from a more general result
in Section 4, where we prove existence of a Morse function on a sphere with $k$ holes,
such that the function is constant on each connected component of the boundary
and the length of fibers are bounded in terms of area and boundary length.

Alvarez Paiva, Balacheff and Tzanev \cite{ABT} show that existence of a Morse function on a
Riemannian 2-sphere with bounded fibers
yields a length-area inequality for the shortest periodic geodesic on a Finsler 2-sphere
(for both reversible and non-reversible metrics). Note that arguments used by Croke \cite{C}
to prove the length-area bound for the shortest closed geodesic on a Riemannian 2-sphere (see also \cite{R}
for the best known constant) can not be directly generalized to the Finsler case
because co-area inequality fails for non-reversible Finsler metrics. 

Hence, Theorem \ref{morse} yields a better constant for Theorem VI in \cite{ABT}.
%

\begin{theorem} \label{finsler}
Let $M$ be a Finsler two-sphere with Holmes-Thompson area $A$. 
Then $M$ carries a closed geodesic of
length $\leq 160 \sqrt{A}$.
\end{theorem}

The reason for constants $26$ and $52$ in our theorems is the following.
We obtain the desired slicing of the sphere by repeatedly using the result
of Papasoglu \cite{P} to subdivide the sphere into smaller regions by a curve of length 
at most $2 \sqrt{3}$
times the square root of the area of the region.
At each step the area of the region reduces at least by a factor of $\frac{3}{4}$.
We then assemble these subdividing curves into one foliation with lengths
bounded by the geometric progression
$$\sum_{i=0} ^{\infty} 2 \sqrt{3} (\frac{3}{4})^{i/2}  \sqrt{|M|}= 4\sqrt{3}(2+\sqrt{3}) \sqrt{|M|}
\leq 26 \sqrt{|M|}$$
In the proof of Theorem \ref{morse} some subdividng curves are used twice so 
an additional factor of $2$ appears.

After the first subdivision happens our regions are no longer spheres,
but rather spheres with a finite number of holes.
It may not be possible to find a short simple closed curve
subdividing it into two parts of area $\leq \frac{3}{4}$ of its area.
Instead we may have to use a collection of arcs with endpoints
on the boundary subdividng the region into many pieces, each of small area.
The issue is then how to assemble all of these subdividing curves into one foliation.

The main technical result of this paper (proved in the next section, see Proposition \ref{subdivision})
is that we can always choose these subdividing arcs in such a way that
they belong to a single connected component of the boundary of a certain subregion
$A_1$ with area of $A_1$ between $\frac{1}{4}$ and $\frac{3}{4}$
of the area of the region. This result make assembling curves into one foliation 
a straightforward procedure.

After the first version of this article appeared on the web,
F. Balacheff in \cite{B} improved constant $26$ in Theorems 1 and 2 
to $7.6$, and constant $160$ in Theorem 4 to $31.1$,
but not constant $52$ in Theorem 3.

\vspace{0.2in}

\noindent
\textbf{Acknowledgements.} 
The author is grateful to Alexander Nabutovsky for reading 
the first draft of this paper and to Juan Carlos Alvarez Paiva, Florent Balacheff,
Gregory Chambers and Regina Rotman for valuable discussions.
The author would like to thank the referee for helpful comments.
The author acknowledges the support by Natural Sciences and Engineering 
Research Council (NSERC) CGS graduate scholarship, and by Ontario Graduate Scholarship.

\section{Subdivision by short curves}

Let $M$ be a Riemannian 2-sphere.
For $0<r\leq \frac{1}{2}$ let $S_r(M)$ denote the set of simple closed curves on $M$
that divide it into subdiscs of area $\geq r |M|$.
Define $c(M,r) = \inf _{\gamma \in S_r(M)}|\gamma|$ and $c(r)= \sup c(M,r)$, where the
supremum is taken over all metrics on $S^2$ of area $1$.

By definition $c(r)$ is increasing and by Proposition 
\ref{semicontinuity} below it
is lower semicontinuous. For any $r> \frac{1}{3}$ it follows from the example in the introduction
(see Figure \ref{starfish}) that $c(r) = \infty$. For $r= \frac{1}{4}$ we have the following result of 
Papasoglu.

\begin{theorem} \label{papasoglu}
{(Papasoglu, \cite{P})} $c(\frac{1}{4}) \leq 2 \sqrt{3}$
\end{theorem}

We will need to generalize Theorem \ref{papasoglu} to spheres with finitely many holes and allow a larger class 
of subdividing curves than just simple closed curves. When the surface has boundary we will allow 
the subdividing curve $\gamma$ to consist of several arcs with endpoints on the boundary.
In this case we will define a distinguished connected component $A_1$ of $M \setminus \gamma$
and require that $\gamma$ is contained in a connected component of $\partial A_1$. This is 
a technical condition that will make it easer to repeatedly cut the surface into smaller pieces 
and concatenate the subdivding curves to obtain a slicing of $M$.

Let $M_k$ be a complete Riemannian 2-surface with boundary homeomorphic to a sphere
with $k$ holes. 
Let $\gamma$ be a simple closed curve in the interior of $M_k$
or a union of finitely many arcs $\gamma= \bigcup \gamma_i$, where $\gamma_i$
are arcs with endpoints on $\partial M_k$ that do not pairwise intersect and have no 
self-intersections. Let $\{ A_i \}$ be the set of connected components of $M_k \setminus \gamma$.
Let $S_{M_k}(r)$ denote the set of all such $\gamma$ on $M_k$ that in addition satisfy

\begin{enumerate}
\item
$r |M_k| \leq |A_1| \leq (1-r) |M_k|$

\item
$\gamma$ is contained in a connected component of $\partial A_1$
\end{enumerate}

In particular, (1) implies that the area of every connected component of $M_k \setminus \gamma$
is bounded from above by $(1-r) M_k$. 

Define $c(M_k,r) = \inf _{\gamma \in S_r(M_k)}|\gamma|$ and $c_k(r)= \sup c(M_k,r)$, where the
supremum is taken over all metrics on a sphere with $k$ holes that have area $1$.

We have the following useful fact.

\begin{proposition} \label{semicontinuous}
$c_k(r)$ and $c(r)$ are lower semi-continuous for $r \in (0, \frac{1}{3}]$.
\end{proposition}

\begin{proof}

\begin{figure} 
   \centering	
	\includegraphics[scale=0.5]{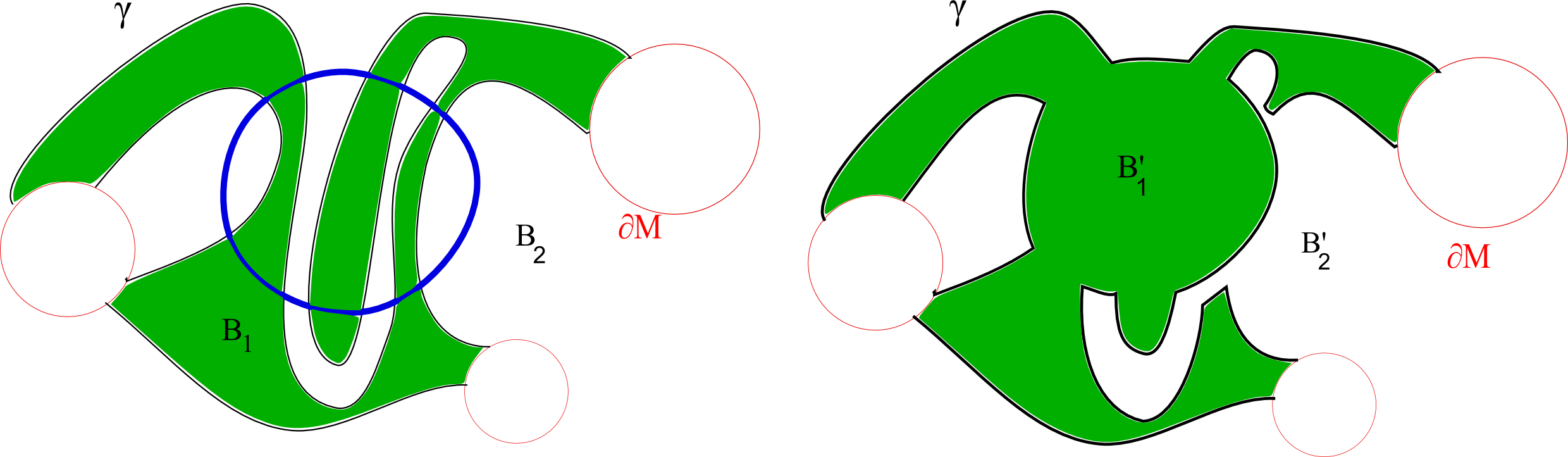}
	\caption{``Fattening" the smaller region} \label{semicontinuity}
\end{figure}

We prove the result for $c_k(r)$ and for $c(r)$ it will follow as a special case 
from the argument below.

Let $\{r_n\}$ be an increasing sequence converging to $r$ for some $0<r\leq \frac{1}{3}$.
Fix $\epsilon >0$. Let $M_k$ be a complete Riemannian surface of area $1$ diffeomorphic to
a sphere with $k$ holes. We would like to show that for some $r_n$ there exists
$\gamma \in S_r(M_k)$ with $|\gamma|\leq c_k(r_n)+\epsilon$.

We can find $\delta>0$ small enough so that it satisfies the following requirements:

\begin{enumerate}
\item
The area of the $\delta$-tubular neighbourhood of $\partial M$
satisfies $|N_{\delta}(\partial M_k)|< \frac{r}{100}$.

\item
Any ball $B_{\delta}(x)$ around $x \notin N_{\delta}(\partial M_k)$
is bilipschitz diffeomorphic to the Euclidean 
disc of radius $\delta$ with Lipschitz constant between $0.99$ and $1.01$. 
In particular, $3 \delta ^2 < |B_{\delta}(x)|< 4 \delta^2$
and $|\partial B_{\delta}(x)| < 8 \delta$.

\item
$\delta < \frac{\epsilon}{100}$
\end{enumerate}

For such a $\delta$ choose $r_n$ so that $r-r_n<(1-r) \delta ^2$.
Let $\gamma \in S_{M_k}(r_n)$ be the subdividing arcs of length $\leq c_k(r_n) + \epsilon/2$.
In the following argument it will be more convenient to consider $\gamma'$, the connected
component of $\partial A_1$ that contains $\gamma$
(recall that by definition of $c_k(r)$ 
$\gamma$ lies in 
a single boundary component of $\partial A_1$). If $\gamma$ is a closed curve then $\gamma' = \gamma$.
If $\gamma$ is a union of arcs then $\gamma'$ is a closed curve made out of arcs of $\gamma$ and arcs of 
$\partial M$.

Let $B_1$ denote the element of $\{A_1, M \setminus A_1 \}$ of smaller area and 
$B_2$ denote the element of larger area. We can assume that
$|B_1|< r$ for otherwise we are done. Therefore, we have $r_n \leq |B_1|<r$
and $1-r< |B_2|\leq 1-r_n$.

Let $M'$ denote $M \setminus N_{\delta}(\partial M_k)$. 
By our choice of $\delta$ we have that the area of a ball $|B_{\delta}(x)|\geq 3 \delta^2$
for $x\in M'$ and $|M' \cap B_2|\geq 0.99 |B_2|$.
By Fubini's theorem we obtain 
$\int _{M'} |B_{\delta}(x) \cap B_2| = |B_2 \cap M'| \int _{M'} |B_{\delta}(x)| \geq 3/2(1-r) |M'| \delta^2$.
Hence, for some $x \in M'$ we have $|B_{\delta}(x) \cap B_2| \geq 3/2(1-r) \delta^2$.
Since $M' \cap B_1$ is non-empty we can always find such a ball so that $\gamma' \cap B_{\delta}(x)$ is non-empty.

We will now construct a new curve $\beta$ that coincides with $\gamma'$ outside of $B_{1.1\delta}(x)$
and divides $M_k$ into regions $B_1'$ and $B_2'$ so that one of them is connected and each of them has area $\geq r$.
Moreover, $|\beta| < |\gamma'| + \epsilon/2$. This implies the desired inequality 
$c(r,M_k)\leq c(r_n,M_k)+ \epsilon$.

We construct $\beta$ by cutting $\gamma'$ at the points of intersection with $\partial B_{\delta}(x)$
and attaching arcs of $\partial B_{\delta}(x)$. We do it in such a way that $B_{\delta}(x)$ is now
entirely contained in the smaller of two regions. This increases the area of the smaller region
by at least $(1-r) \delta^2$ and increases the length of the subdividing curve by at most $2|\partial B_{\delta}(x)|
\leq 16 \delta < \epsilon/2$.

The procedure is illustrated in Figure \ref{semicontinuity}. 
Let $\{C_j \}$ be connected components of $B_1 \setminus B_{\delta}(x)$ and $\{c_i ^j\}$ denote connected 
components of $\partial C_j \cap \partial B_{\delta}(x)$. 
First we erase all arcs of $\gamma'$ that are in $B_{\delta}(x)$. For each $j$ we erase the arc $c^j _1$
from $\partial B_{\delta}(x)$. For each $c^j _i$, $i > 1$, we add a copy of $c^j _i$ and 
perturb it so that the new curve $\beta$ does not intersect $B_{\delta}(x)$ in the neighbourhood
of $c^j _i$. This does not increase the number of connected components of either region.

\end{proof}

The only example that I know where upper semicontinuity of $c(r)$ fails is the three-legged starfish
on Figure \ref{starfish}
(showing $c(r) = \infty$ for $r> \frac{1}{3}$, while $c(1/3) \leq 26$ by Theorem \ref{1/3 area}).
The following question seems natural: 

\vspace{0.1in}
\noindent
\textbf{Question} Is $c(r)$ continuous for $r \in (0,\frac{1}{3})$?
\vspace{0.1in}

It can be easily shown from the definition that 
$c(r) \leq c_k(r)$. Under the additional assumption
$r\leq \frac{1}{4}$ we are able to prove that they are
equal.

\begin{proposition} \label{subdivision}
Suppose $r\leq \frac{1}{4}$, then $c_k(r)= c(r) $.
\end{proposition}

In the proof of Proposition \ref{subdivision} we will use the following topological fact.

\begin{lemma} \label{2-gon}
Let $M$ be a submanifold (with boundary) of $S^2$ and let $\gamma$ be a simple closed curve in $S^2$.
Suppose the intersection of $\gamma$ and $\partial M$ is non-empty and transversal.
If $A$ denotes a connected component of $S^2 \setminus \gamma$
then there exists an arc $a \subset \gamma$ and an arc $b \subset \partial M$, such that
$a \cup b$ bounds a disc $D_{a \cup b} \subset A$ and
$D_{a \cup b} \cap M$ is connected.
\end{lemma}

\begin{figure} 
   \centering	
	\includegraphics[scale=1]{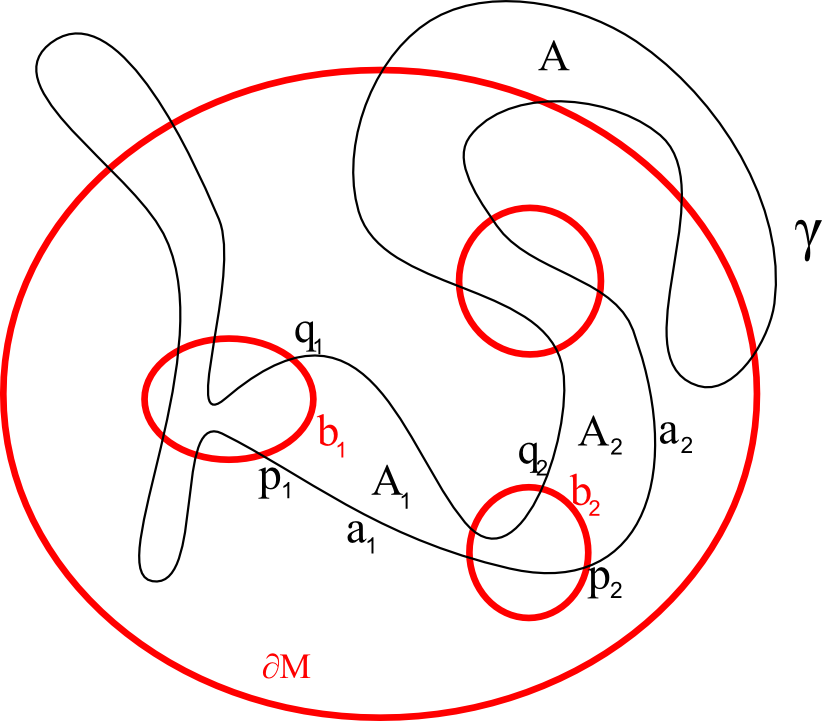}
	\caption{Proof of Lemma \ref{2-gon}} \label{2gon_fig}
\end{figure}

The proof of this 
Lemma is the main technical part of the paper.
The proof is illustrated on Figure \ref{2gon_fig}.
The reader can draw several pictures like on Figure \ref{2gon_fig}
and convince him or herself that the statement is correct.

\begin{proof}
The strategy of the proof is to find some disc $D_{a' \cup b'}$,
with possibly a large number of components of $D_{a' \cup b'} \cap M$.
We then show that there exists a subdisc of $D_{a' \cup b'}$ whose intersection
with $M$ has a smaller number of connected components.

Let $p_1 \in \partial M \cap \gamma$
and $C_1$ be the connected component of $\partial M$
that contains $p_1$.

Define an interval $b_1 \subset C_1$
to be a connected component of $C_1 \cap A$
that has $p_1$ as an enpoint
(note that it is unique). 
We denote the other endpoint of $b_1$ by 
$q_1$.
Let $v$ be a tangent vector to $\gamma$ at $p_1$
pointing inside $M$. 
Let $a_1$ be an arc of $\gamma$ that starts at
$p_1$ in the direction $v$ and ends at $q_1$.

Observe that the curve
$a_1 \cup b_1$ is a simple closed curve enclosing a disc $D_{a_1 \cup b_1} \subset A$ 
that contains a non-empty subset of $M \cap A$.
If $\overline{A}_1 = D_{a_1 \cup b_1} \cap M$
has one connected component then we are done.
Assume it has more than one.
Let $A_1$ be a component of $\overline{A}_1$ with $b_1 \subset \partial A_1$
and let $A_2$ be a different component. Define a point $p_2 \in a_1$ by
$$p_2 = a_1(\inf _{0 \leq t \leq 1} \{ t | a_1(t) \in \partial A_2 \}) $$
It follows from the definition that $p_2 \in \partial M$.
We can find a point $q_2 \in \gamma \cap \partial M$
and an arc $b_2 \subset \partial M \cap A$ from $p_2$ to $q_2$.

The interior of $\gamma \setminus a_1$ 
is contained in the interior 
of $S^2 \setminus D_{a_1 \cup b_1}$. 
It follows that $q_2 \in a_1$.
Denote the arc of $a_1$ between $p_2$ and $q_2$ by $a_2$.
We have that $a_2 \cup  b_2$ separates $S^2$ into a disc $D_{a_2 \cup b_2}$ that contains
$A_2$ and its complement that contains $A_1$.
Set $\overline{A}_2 = D_{a_2 \cup b_2} \cap A$.
It is non-empty and has fewer connected components than $\overline{A}_1$.
We iterate this procedure until we are left with just one connected component.

\end{proof}

We now prove Proposition \ref{subdivision}.

The direction $c(r) \leq c_k(r)$ is simple.
Given a Riemannian 2-sphere we can make $k$ holes in it of small area and small boundary length.
We subdivide the resulting sphere with holes and
use the fact that the connected component of $\partial A_1$
that contains $\gamma$ will have length close 
to $c_k(r)$.
Here we did not use that $r \leq \frac{1}{4}$.

To prove the other direction we proceed as follows.
Let $\epsilon$ be a small positive constant.
Given a sphere with $k$ holes $M_k$ of area $1$ we attach $k$ discs of total area $\epsilon$ to the boundary of $M_k$.
We obtain a Riemannian 2-sphere $M$. 

Let $\gamma_1$ be a simple closed curve of length $\leq c_r(M) +\epsilon$ 
subdividing $M$ into two subdiscs $A$ and $B$ of area
between $r(1+\epsilon)$ and $(1-r)(1+\epsilon)$. 
Suppose $A$ is a subdisc of area $\geq \frac{1}{2}(1+\epsilon)$.
If $\gamma_1$ is disjoint from $\partial M_k$ 
we conclude that $c(M_k,\frac{r}{1+\epsilon})\leq c_r(M)+ \epsilon$.

Suppose $\gamma_1$ intersects $\partial M_k$.
By Lemma \ref{2-gon} there exists an arc $a_1$ of $\gamma_1$ and an arc $b_1$ of $\partial M_k$,
such that $A_1=D_{a_1 \cup b_1} \cap M_k \subset A$ is connected. 

We consider two possibilities. First, suppose $|A_1| \geq r(1+\epsilon)$. 
Let $\gamma_2$ be the intersection of $a_1$ with the interior of $M_k$.
$A_1$ is a connected component of $M_k \setminus \gamma_2$ of area
between $r-\epsilon$ and $1- r +\epsilon$.  The rest of connected 
components of $M_k \setminus \gamma_2$ have area less than $1- r +\epsilon$.
So $c(M_k,r-\epsilon)\leq c_r(M) +\epsilon$.

Alternatively, suppose $|A_1| < r(1+\epsilon)$. 
In this case we can define a new curve $\gamma_2$, such that the number of connected components of
$M_k \setminus \gamma_2$ is smaller than the number of connected components of $M_k \setminus \gamma_1$.
We do this by replacing the arc $a_1$ of $\gamma_1$ by $b_1 \subset \partial M_k$.
Note that we can slightly perturb the part of the new curve that coincides with $b_1$ 
so that it is entirely in $M \setminus M_k$, in particular,
the intersection of $\gamma_2$ with $\partial M_k$ is transversal and
the length of the intersection of $\gamma_2$ with the interior of $M_k$
is smaller than that of $\gamma_1$. 
As a result we transferred the area of $A_1$ from $A$ to $B$.
Since $|A_1|< r(1+\epsilon)$, $r\leq \frac{1}{4}$ and $|A| \geq \frac{1+\epsilon}{2}$
we obtain that the area of each of the two discs $M \setminus \gamma_2$ is at least $r(1+\epsilon)$. 
In this way we can continue reducing the number of connected components of $M_k \setminus \gamma_i$ until one of the 
subdiscs of $M \setminus \gamma_i$ contains only one connected component of $M_k$ or until
we encounter the first possibility above.

By Proposition \ref{semicontinuous} we conclude that $c_k(r) \leq c(r)$.

\section{$T$-maps}

\begin{definition}
A map $f$ from $M_k$ to a trivalent tree $T$  is called a $T$-map if the topology of fibers of $f$
is controlled in the following way:
the preimage of any point in an edge of $T$ is a circle, there exist $k$ terminal vertices $x_k \in T$,
such that $f^{-1}(x_k)$ is a connected component of $\partial M_k$, the preimage of any other terminal point of
$T$ is a point, and the preimage of a trivalent vertex of $T$ is homeomorphic to the greek
letter $\theta$. 
\end{definition}

\begin{theorem} \label{tree1}
For $r \in (0,\frac{1}{4}]$ and any $\epsilon >0$ there exists a $T-$map $f$ from 
$M_p$, $p \geq 0$, so that each fiber of the map has
length less than $\frac{c(r)}{1-\sqrt{1-r}} + |\partial M_p| + \epsilon$.
\end{theorem}

Theorem \ref{tree} follows by taking $r=\frac{1}{4}$ and applying Theorem \ref{papasoglu}.

In the proof we will repeatedly use the following simple fact,
so it is convenient to state it as a separate lemma.

\begin{lemma} \label{concatenation}
Let $A_1$ and $A_2$ be two closed smooth submanifolds (with boundary) of $M_p$, 
such that $\alpha= A_1 \cap A_2$ is a connected arc.
Let $c_i$ denote the connected component of $\partial A_i$ that
contains $\alpha$.
Suppose $|c_1 \cup c_2|< L$ and that each $A_i$ admits a $T-$map with fibers
of length $< L$, then $A_1 \cup A_2$ admits a
$T-$map with fibers of length $\leq L$.
\end{lemma}

\begin{proof}
\begin{figure} 
   \centering	
	\includegraphics[scale=0.2]{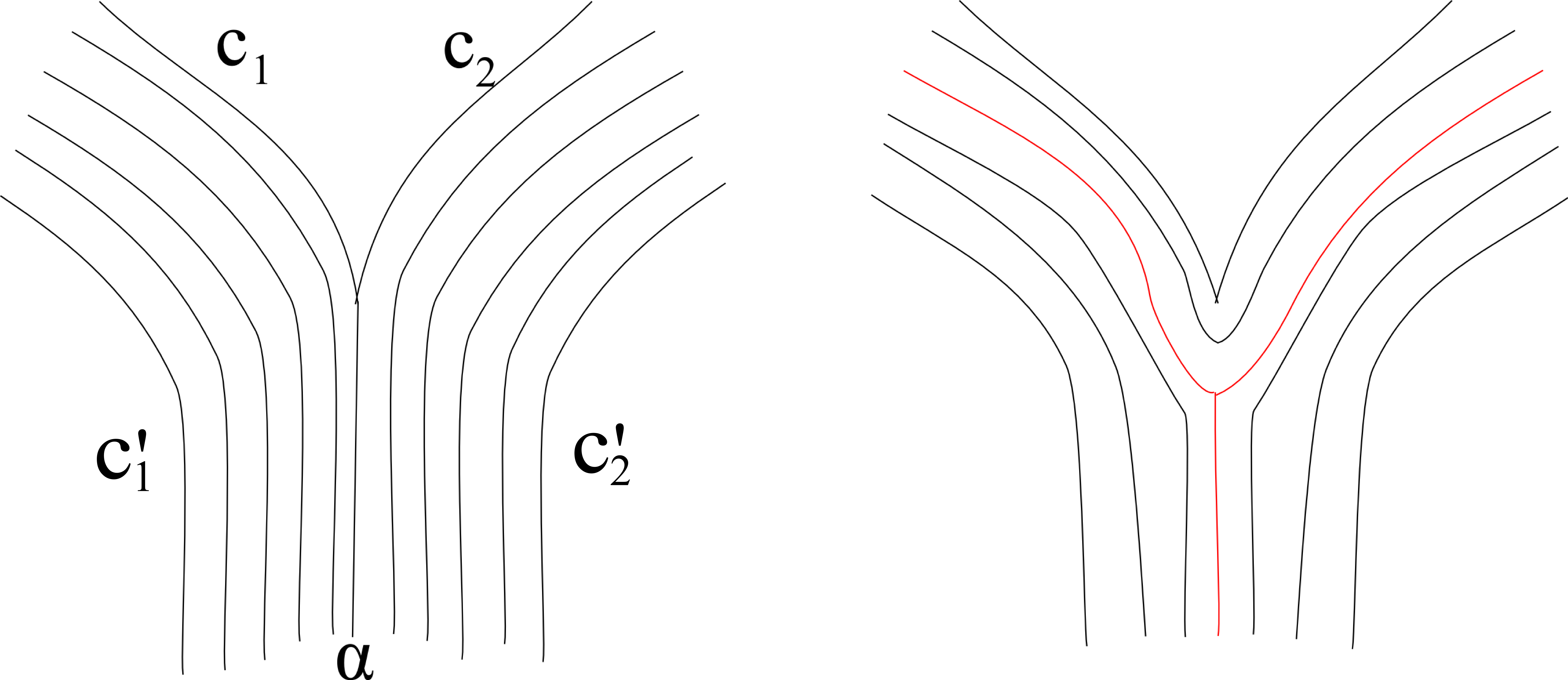}
	\caption{Surgery in the neighbourhood of $a_i$} \label{concat_fig}
\end{figure}


From the assumption that $A_i$ admits a $T-$map it follows that
there exists an embedded cylinder $C_i \subset A_i$ and a map $f_i:C_i \rightarrow [0,1]$
with fiber $f^{-1}(0)=c_i$ and the length of all fibers $< L$.
The boundary $\partial C_i = c_i \cup c_i'$ with $c_i'$ contained in the
interior of $A_i$. 

Let $a_1$ and $a_2$ be the endpoints of $\alpha$.
For a sufficiently small $\delta>0$ we perform a surgery on the the closed curves in
 the foliation $\{f_i^{-1}(t)|0 \leq t \leq \delta \}$.
 The surgery happens in the $\delta-$neighbourhood of $a_i$
 and is depicted on Figure \ref{concat_fig}.
 
 As a result of the surgery we obtain three families of curves.
 The ``outer" family converging to $c_1 \cup c_2$, and two ``inner" families each converging
 to $c_1'$ or $c_2'$.
 These three families are separated by a $\theta$ graph (drawn in red on Figure \ref{concat_fig}).
 This surgery defines the desired $T-$map.
 
\end{proof}

We will need first a version of Theorem \ref{tree1} for very small balls.

\begin{lemma} \label{small length}
For any $\epsilon > 0$ there exists $l>0$, such that for every disc $D \subset M_p$
with $|\partial D| \leq l$ there exists a diffeomorphism $f$ from $D$
to the standard closed disc $D_{st}=\{x^2+y^2 \leq 1 \}$
so that the preimage of each concentric circle $f^{-1}(\{x^2 +y^2 = const \})$, 
 has length $\leq (1+\epsilon)|\partial D|$.
\end{lemma}

\begin{proof}
For $\delta>0$ sufficiently small 
every ball $B \subset M_p$ of radius $\delta$
is $(1+\epsilon)-$bilipschitz diffeomorphic to a disc
in the closed upper half-plane $\mathbb{R}^2_+$. 

Let $D'$ denote the image of 
$D$ under such a diffeomorphism.
(Here we are assuming that $l<\delta$ and $D$ is contained within 
a ball of radius $\delta$).
After a small perturbation we may assume that 
the projection $p$ of $\partial D'$ onto $y$ coordinate is a Morse 
function. 

Define a 1-dimensional simplicial complex $T \subset D'$ as follows. 
Let $a$ be a regular value of the projection function $p|_{\partial D'}$
restricted to the boundary of $D'$. $p^{-1}(a) \cap D'$ is a finite union of disjoint 
closed intervals $\{v_i\}$. We set $T \cap p^{-1}(a)$ to be the midpoints 
of $v_i'$s. If $a$ and $b$ are two consecutive critical values of $p$,
it follows that $T \cap p^{-1}(a,b)$ is a collection of disjoint simple arcs as on 
Figure \ref{fig_small length}.

At a critical point $\partial D'$ locally looks like the graph of 
a function $f(x)=\pm x^2$.
We connect the endpoints of the intervals of $T$
by a horizontal arc tangent to the critical point and 
contained in $D'$. 

Note that $D'$ retracts onto $T$, so in particular $T$ must be connected and simply connected,
hence a tree. We contract $D$ along the edges of $T$ in the obvious way.
As a result we obtain a contraction of $\partial D$ inside $D$ to a point through curves 
of length $\leq (1+\epsilon)|\partial D|$. After a small perturbation we
can assume that this homotopy realizes the desired diffeomorphism.
For details we refer the reader to \cite{CR}, where it is shown that
if there exists a homotopy of the boundary of $D$ to a point through 
curves of length $<L$, then there exists a diffeomorphism from 
$D$ to $D_{st}$ so that preimages of concentric circles have length $<L$.
\begin{figure} 
   \centering	
	\includegraphics[scale=0.3]{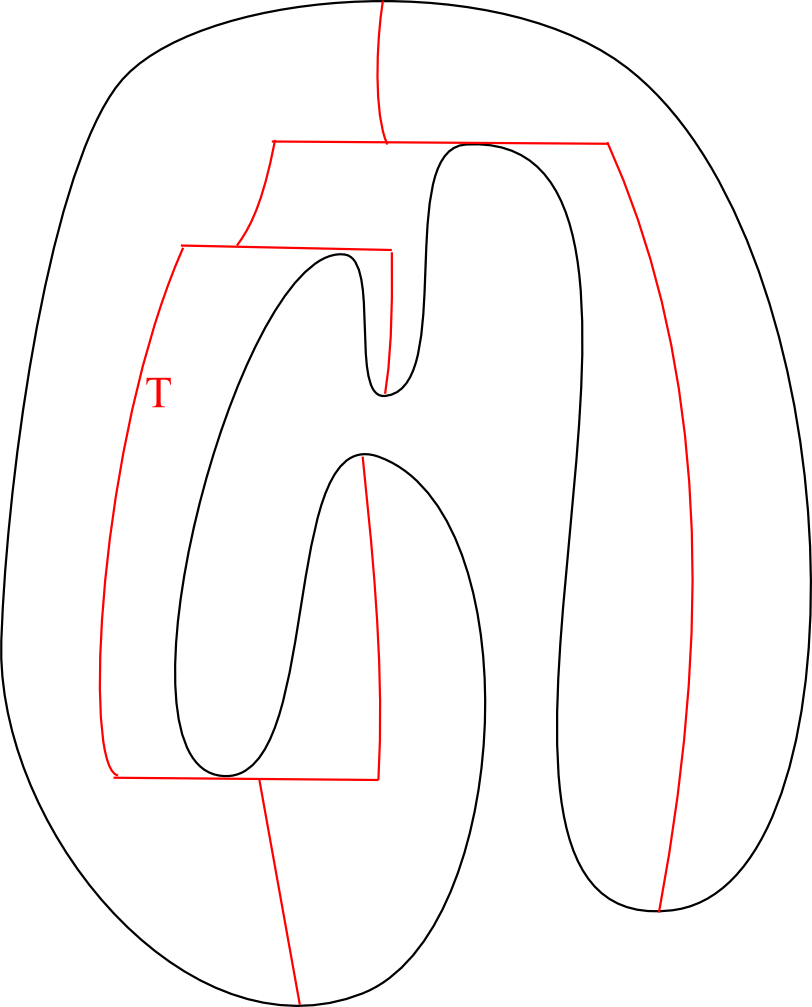}
	\caption{Tree $T$ in $D'$} \label{fig_small length}
\end{figure}

\end{proof}  

\begin{lemma} \label{small area_S^2}
For any $\epsilon > 0$ there exists $A>0$, such that for every disc $D \subset M_p$
with $|D| \leq A$ there exists a $T-$map $f$ from $D$ with fibers of
length less than $|\partial D| + \epsilon$. 
\end{lemma}

\begin{proof}
The proof is similar to that of Lemma 2.2 in \cite{LNR}.

Choose $A < \min \{ \epsilon, \frac{l^2}{64} \}$, where $l$ is as in Lemma \ref{small length}.
The proof is by induction on $n=\lceil \frac{|\partial D|}{\sqrt{|A|}} \rceil$.
For $n\leq 8$ the result follows by Lemma \ref{small length}.
Assume the Lemma to be true for all subdiscs with $\lceil \frac{|\partial D|}{l} \rceil<n$
and consider the case when this quantity equals $n \geq 8$.

Subdivide $\partial D$ into 4 arcs of equal length.
By Besicovitch Lemma we can find an arc $\alpha$ of length $\leq \sqrt{|A|}$ connecting two opposite 
arcs. $\alpha$ subdivides $D$ into two subdiscs $D_1$ and $D_2$ of area $\leq A$ and
boundary length $\leq \frac{3}{4}|\partial D|+ \sqrt{|A|}\leq (n-1) \sqrt{|A|}$.
Hence, by inductive assumption $D_1$ and $D_2$ admit $T-$maps with fibers of length
$\leq |\partial D|$.

By inductive assumption each disc admits a $T-$map with fibers of length $\leq |\partial D|+\epsilon$.
We also have $|\partial D_1 \cup \partial D_2| \leq |\partial D|+ \epsilon$ so the result follows by Lemma \ref{concatenation}.

\end{proof}

We need to generalize this result about small discs to other small submanifolds of $M_p$.

\begin{lemma} \label{small area_Mk}
For any $\epsilon > 0$ there exists $A>0$, such that for every submanifold with boundary $M_k \subset M_p$,
with $|M_k| \leq A$ there exists a $T-$map $f$ from $M_k$ with fibers of length
$\leq |\partial M_k|+4(k-1) \sqrt{A}+\epsilon$.
\end{lemma}

\begin{proof}
Let $M_k$ be a closed submanifold (with boundary) of $M_p$ and let $c$
be a connected component of $\partial M_k$. If $dist(c,\partial M_k \setminus c)=d$
then there exists an open ball $B(d/2) \subset M_k $ of radius $d/2$
whose interior does not intersect the boundary of $M_k$ (and, in particular, it
does not intersect the boundary of $M_p$). As $d \rightarrow 0$
the area of $B(d/2)$ approaches the area of a Euclidean disc
of the same diameter. Since $M_p$ is compact, 
this happens uniformly for all balls of radius $d/2$ disjoint from the boundary.
 For $d$ sufficiently small we may conclude
that $|M_k|\geq 3 (d/2)^2$.
Hence, for a sufficiently small $A$, if $|M_k|\leq A$ then the distance 
$dist(c,\partial M_k \setminus c)\leq \frac{2 \sqrt{3}}{3} \sqrt{A}$.
We attach $k-1$ arcs $\{ \gamma_i \}$ to the boundary of $\partial M_k$
of total length less than $2(k-1)\sqrt{A}$
and so that $\partial M_k \cup \bigcup \gamma_i$ is connected and its complement
in $M_k$ is homeomorphic to a disc. Denote this disc by $D$.

Consider the normal $\delta-$neighbourhood $N_{\delta}$ of $\partial D$
in $M_k$ (see Figure \ref{fig_small area}) for some small $\delta$. By Lemma \ref{small area_S^2}
the complement of $N_{\delta}$ in $M_k$ admits a $T$ map with fibers of length 
$\leq |\partial M_k|+4(k-1) \sqrt{A}+\epsilon$.
Let $\alpha$ be a short closed curve in $N_{\delta}$ which separates one
connected component of $\partial M_k$ from other connected components as
on Figure \ref{fig_small area}. Curve $\alpha$ separates $N_{\delta}$ into two regions.
Let $B$ denote the region that contains only one connected component of $\partial M_k$.
By Lemma \ref{concatenation} we can extend the $T-$map to $D \cup B$.
By chopping off connected components of $\partial M_k$ and applying Lemma \ref{concatenation} 
repeatedly we obtain the desired $T-$map.

\begin{figure} 
   \centering	
	\includegraphics[scale=0.45]{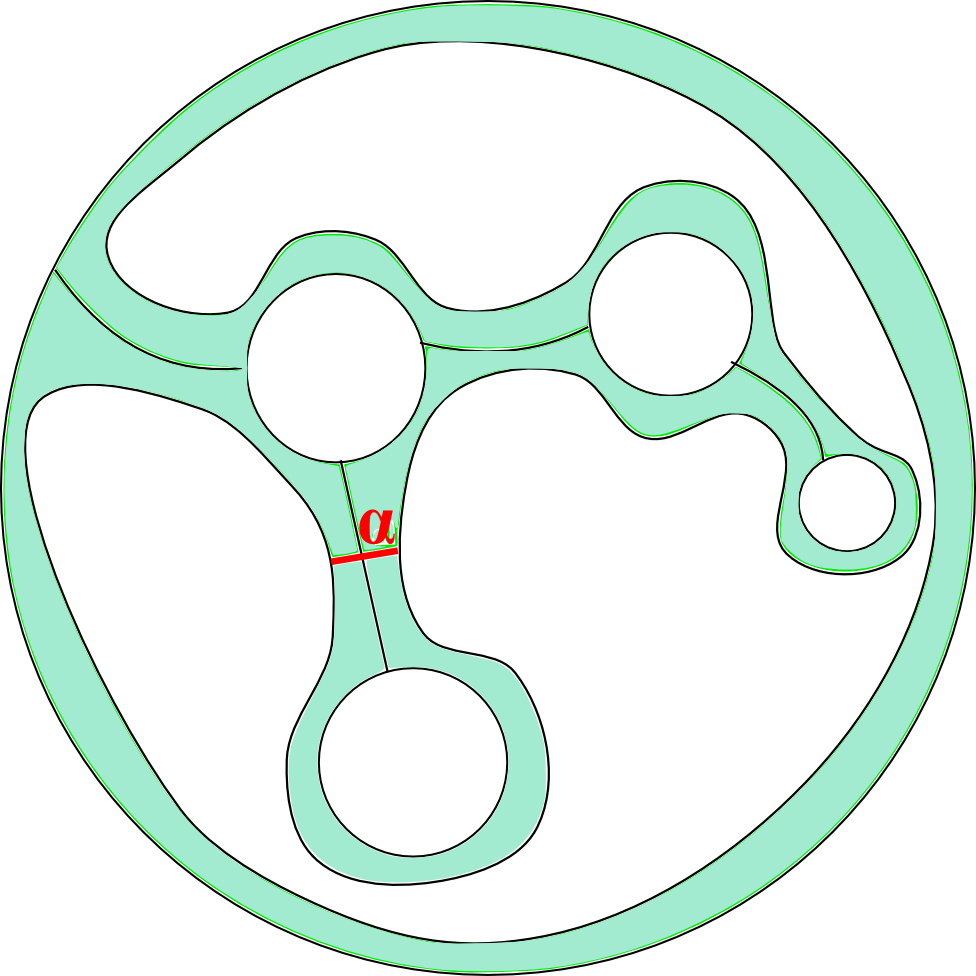}
	\caption{Inductive step of the proof of Lemma \ref{small area_Mk}} \label{fig_small area}
\end{figure}

\end{proof}

The proof of Theorem \ref{tree1} proceeds inductively by cutting $M$ into smaller pieces 
until their size is small enough so that Lemma \ref{small area_Mk} can be applied.
We assemble $T-$maps on these smaller regions to obtain one map from $M$
with the desired bound on lengths of fibers.

Fix $\epsilon>0$ and let $A$ be as in Lemma \ref{small area_Mk}.
Let $N= \lceil \log_{4/3}(\frac{|M|}{A}) \rceil$.
We claim that for every $M_k \subset M_p$ with $(\frac{3}{4})^{n+1}|M_p|<|M_k|\leq (\frac{3}{4})^{n}|M_p|$
there exists a $T-$map with fibers of length 

\begin{equation}\label{inductive statement}
\leq |\partial M_k| + 4 (N-n+k)\sqrt{A}+ \sum_{i=n}^{N-1} c(r) (\frac{3}{4})^{i/2} \sqrt{|M_k|}+\epsilon
\end{equation}

When $n=N$ the inequality (\ref{inductive statement}) is true by Lemma \ref{small area_Mk}.
We assume it to be true for $n+1 \leq N$ and prove it for $n$.

By Proposition \ref{subdivision} there exists
$\gamma \in S(M_k,r)$ of length $\leq c(r) + \epsilon'$.

We have two possibilities.

\textbf{Case 1.} $\gamma$ is a simple closed curve.

In this case $\gamma$ separates $M_k$ into two regions $N_1$ and $N_2$ of area
 $|N_i|\leq (\frac{3}{4})^{n+1}|M|$.
 The number of connected components of $\partial N_i$
is at most $k+1$.
 
By inductive assumption $N_1$ and $N_2$ admit $T-$maps into trees
$T_1$ and $T_2$ respectively. We construct $T$ by identifying 
the terminal vertex $f_1(\gamma)$ of $T_1$ with
the terminal vertex $f_2(\gamma)$ of $T_2$. $T-$map $f$
is defined by setting it equal to $f_i$ when restricted to $N_i$.
A simple calculation shows that lengths of fibers of $f$ satisfy 
the desired bound.

\textbf{Case 2.} $\gamma$ is a collection of arcs with endpoints on $\partial M_k$.

Let $\{S_1,...,S_j \}$ be connected components of $\partial M_k$
that intersect arcs of $\gamma$. For a sufficiently 
small $\delta$ the normal neighbourhood $N_{\delta}(S_i)$
is foliated by closed curves of length very close to $|S_i|$,
which are transverse to the arcs of $\gamma$.
In particular, each $N_{\delta}(S_i)$ admits a $T-$map with 
fibers of length $\leq |S_i|+\epsilon'$.
Let $A_1$ be the connected component of $M \setminus \gamma$
with area satisfying $r |M| \leq |A_1|\leq (1-r) |M|  $ (recall that it exists by definition of $c_k(r)$).
Let $B$ denote $A_1 \setminus \bigcup N_{\delta}(S_i)$
and $M_k'$ denote $M_k \setminus N_{\delta}(S_i)$. (See Figure \ref{fig_horseshoe}).

\begin{figure} 
   \centering	
	\includegraphics[scale=0.25]{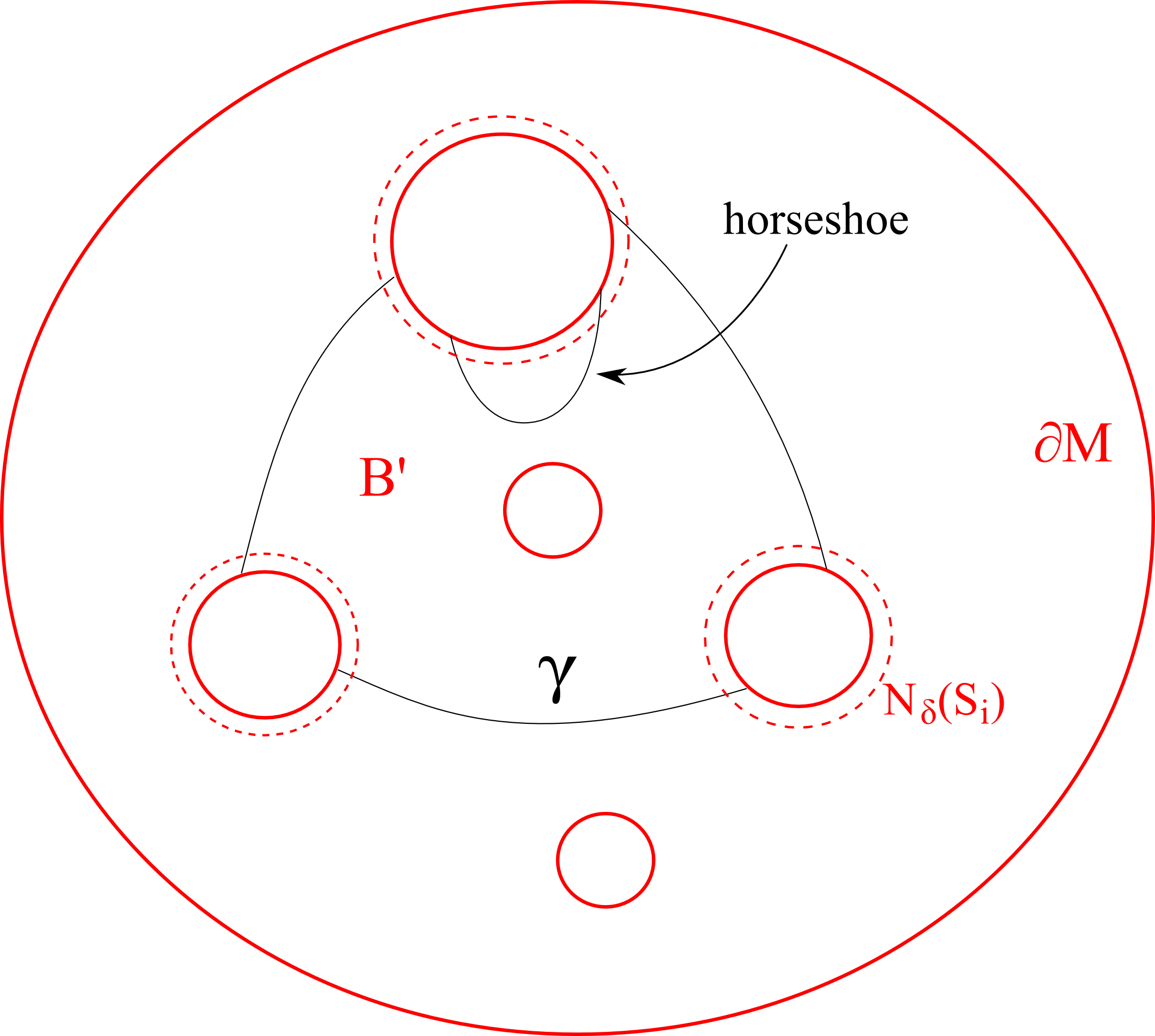}
	\caption{Case 2 in the proof of Theorem \ref{tree1}} \label{fig_horseshoe}
\end{figure}

Let $C$ be the connected component of $\partial B$ that contains
$\gamma \cap M_k'$.
We say that an arc $\alpha$ of $\gamma$ is a horseshoe
if $\alpha$ is in $M_k'$ and the endpoints of $\alpha$ lie on the same connected component 
of $\partial M_k'$. We will use the following simple observation.

\begin{lemma} \label{horseshoe}
If $\gamma$ contains no horseshoes, then there exists a
connected component $S$ of $\partial M_k'$, such that 
$S \cap C$ is connected.
\end{lemma}

\begin{proof}
Let $S$ be a connected component of $\partial M_k'$ and suppose $S \cap C$ contains more than one interval.
Let $S'$ be an interval of $S \setminus C$ and let $a$ and $b$ denote the endpoints of $S'$.
$C \setminus a \cup b $ consists of two arcs, call them $C_1$ and $C_2$. Let $S_1 \neq S$ be a connected component
of $\partial M_k'$ that intersects $C_1$. We claim that $C_2$ does not intersect $S_1$.
For suppose it does, consider then a subarc $C_1'$ of $C_1$ from $a$ until the fist point of intersection with 
$S_1$ and denote this point by $a_1$. Similarly, denote by $C_2'$ the subarc of $C_2$ from $a$
until the fist point of intersection with $S_1$ and call it $a_2$. Let $S_1'$ be an arc of $S_1$ from
$a_1$ to $a_2$, then $\beta = C_1' \cup C_2' \cup S_1'$ is a closed curve separating $M_k'$ into two connected components.
Moreover, the point $b$ and points of $(C_1' \cup C_2')\cap S$ belong to different connected components of
$M_k' \setminus \beta$. This is a contradiction since $B$ is connected.


Suppose the intersection $C \cap S_1$ is not connected. We can then find a subarc $C_3 \subset C_1$
that intersects $S_1$ only at the endpoints. The number of connected components of $\partial M_k'$ that intersect $C_3$
is strictly smaller than the number of components that intersect $C_1$.
Proceeding in this way we can find an arc of $C$ that has endpoints on $S_j$, and its interior intersects only one 
connected component $S_{j+1}$ of $\partial M_k'$. It follows that $S_{j+1} \cap C$ is connected.
\end{proof}

We can now construct the desired $T-$map.
By inductive assumption there exists a $T-$map on $B$ with fibers
$\leq |\partial B| + 4 (N-n-1 +k') \sqrt{A} + \sum _{i=n+1} ^{N-1} c(r) (\frac{3}{4})^{i/2} \sqrt{|M|}+\epsilon'$.
Note that $|\partial B|\leq |\partial M_k| + c(r)\sqrt{|M_k|}+\epsilon'$
and the number of connected components $k'$ of $\partial B$ satisfies $k' \leq k$.

Suppose first that $\alpha \subset \partial B$ is a horseshoe. 
It separates $M_k'$ into two connected components. Let $B_2$ denote the connected component
that does not contain $B$. Since $|B_2|\leq (\frac{3}{4})^{n+1}|M| $,
it admits a $T-$map with the desired bound on length of fibiers. 
By Lemma \ref{concatenation} we can extend the $T-$map
to $B \cup B_2$. 
Inductively we extend the $T-$map to a subset $B' \subset M_k'$,
such that $B \subset B'$ and $\partial B'$ does not contain any
horseshoes. By Lemma \ref{horseshoe} $\partial B' \cap N_{\delta}(S_i)$
is connected for some $i$. By Lemma \ref{concatenation} we can extend the $T-$map
to $B' \cup N_{\delta}(S_i)$.
We iterate this procedure until we extended the $T-$map to a set $B''$,
which contains $N_{\delta}(S_i)$ for all $i$. 
The last step is to extend the $T-$map to the whole of $M_k$
exactly as we did it in Case 1.

%

To finish the proof of Theorem \ref{tree1} recall that
$N \sqrt{A}= \lceil \log_{4/3}(\frac{|M|}{A}) \rceil \sqrt{A} \rightarrow 0$
as $A \rightarrow 0$.
Since $A$ can be chosen arbitrarily small Theorem \ref{tree1} 
follows from (\ref{inductive statement}).

Now we can prove Theorem \ref{1/3 area}. Let $M$ be a Riemannian 
2-sphere and $f:M \rightarrow T$ a $T-$map with fibers of length
$\leq 26 \sqrt{|M|}$. If there are no trivalent vertices in $T$
we can find a short closed curve that separates $M$ into two halves of equal area.

For each trivalent vertex $v_k \in T$
let $\alpha_1^k \cup \alpha_2^k \cup \alpha_3^k = \theta_k = f^{-1}(v_k)$.
Denote the closed curve $\alpha^k_i \cup \alpha^k_j$ by $\gamma^k_{ij}$.
Assume that for every $k,i,j$ ($i \neq j$) $\gamma^k_{ij}$ separates
$M$ into two discs, s.t. the area of the smaller disc is strictly smaller
than $\frac{1}{3}|M|$. Let $k,i,j$ be such that the area of the smaller disc is maximized
among all such curves. 

Let $D_s$ denote the smaller and $D_l$ denote the larger of the subdiscs of 
$M\setminus \gamma^k_{ij}$. 
Let $e$ denote the edge of $T$ adjacent to $v_k$, such that $f^{-1}(e) \subset D_{l}$.
We observe that the other two edges adjacent to $v_k$ are contained in $f(D_s)$,
for otherwise it would follow that the area of the 
smaller disc is not maximal for $\gamma^k_{ij}$.

We conclude that for some $x\in e$ $f^{-1}(x)$ subdivides $M$ into two discs
of area $\geq \frac{1}{3}|M|$ for otherwise it would again contradict our choice
of $\gamma^k_{ij}$.

\section{Morse function}

In this section we prove the existence of a Morse function $f$ from
$M$ to $\mathbb{R}$ with short preimages.

\begin{definition} \label{m-function def}
Let $M$ be a manifold with boundary. A function $f:M \rightarrow \mathbb{R}$
is called an $m-$function if it is Morse on the interior of $M$,
constant on each boundary component and  
maps boundary components to distinct points disjoint 
from the critical values of $f$.
\end{definition}

\begin{theorem} \label{morse1}
For $r \in (0,\frac{1}{4}]$ and any $\epsilon >0$ there exists an $m-$map $f$ from 
$M_p$, $p \geq 0$, so that each fiber of the map has
length less than $\frac{2 c(r)}{1-\sqrt{1-r}} + |\partial M_p| + \epsilon$.
\end{theorem}

Theorem \ref{morse} follows by taking $r = \frac{1}{4}$ and applying Theorem \ref{papasoglu}.

The proof of Theorem \ref{morse1} proceeds along the same lines as the proof 
of Theorem \ref{tree1}. There are two main differences. The fist difference is that 
we would like the function to be smooth (in particular, curves in the foliation $f^{-1}(x)$
are not allowed to have corners) and have singularities of Morse type.
This is accomplished by a simple surgery described in Lemma \ref{theta}.

The second difference is that we would like to bound the length
of the whole level set, not just of 
the individual connected components.

We will need one technical lemma similar to concatenation Lemma \ref{concatenation}.
Let $\theta \subset M_p$ denote a union of three non-intersecting simple arcs $\alpha_i$
with common endpoints.
$\theta$ subdivides $M_p$ into three regions $U_{12}$,  $U_{23}$,  $U_{13}$, 
choosing the indices so that $\alpha_i \cup \alpha_j \subset \partial U_{ij}$.
Let $l_{ij}$ denote a curve obtained by pushing $\alpha_i \cup \alpha_j$ inside $U_{ij}$
by a small perturbation, so that it is contained in a small normal neighbourhood of 
$\theta$, smooth and has length $<|\alpha_i|+|\alpha_j| + \epsilon$. Let
$U(\theta)$ denote the neighbourhood of $\theta$ bounded by curves $l_{ij}$
(see Figure \ref{fig_theta}).

\begin{figure} 
   \centering	
	\includegraphics[scale=0.3]{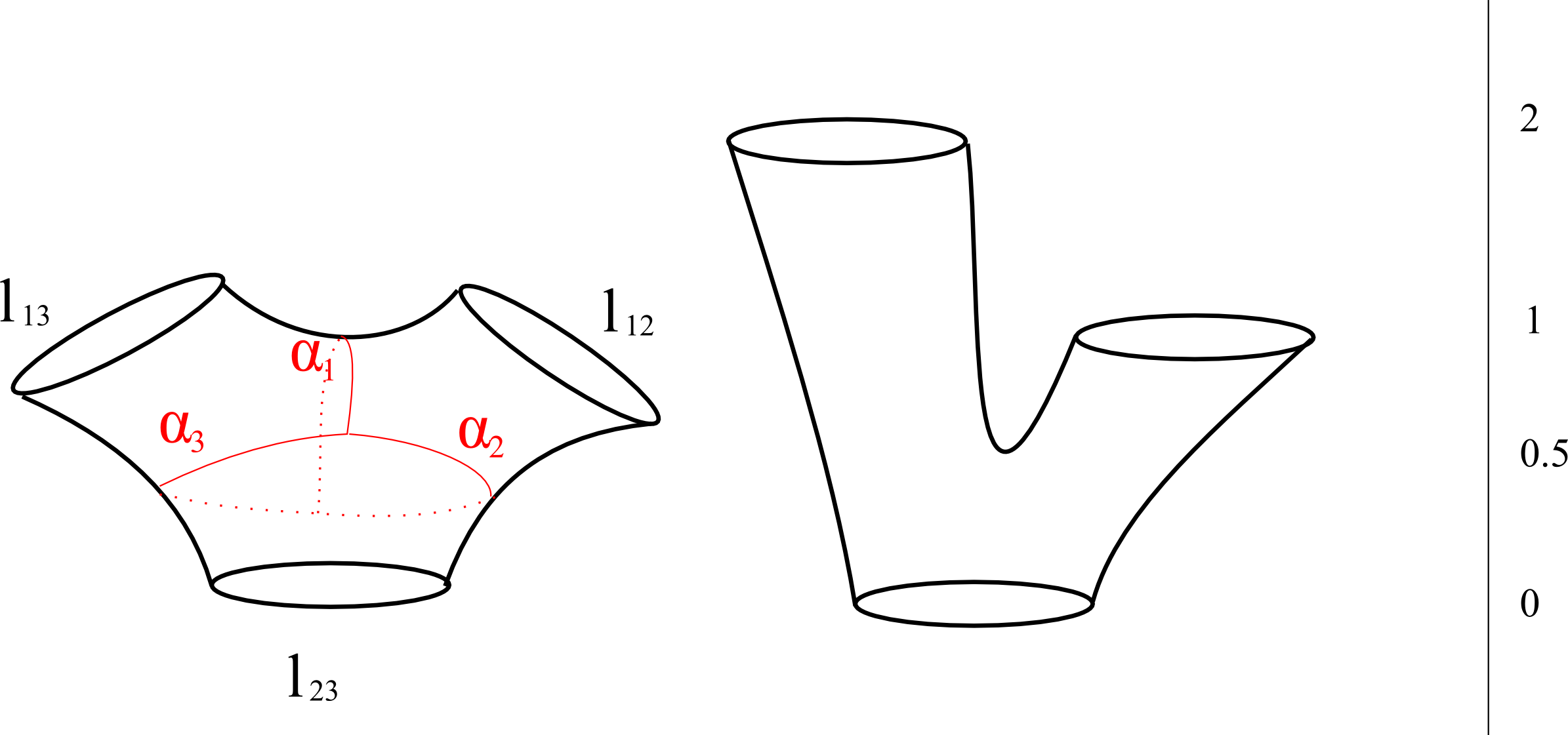}
	\caption{Consturcting an $m-$function on the neighbourhood of $\theta$.} \label{fig_theta}
\end{figure}

\begin{figure} 
   \centering	
	\includegraphics[scale=0.7]{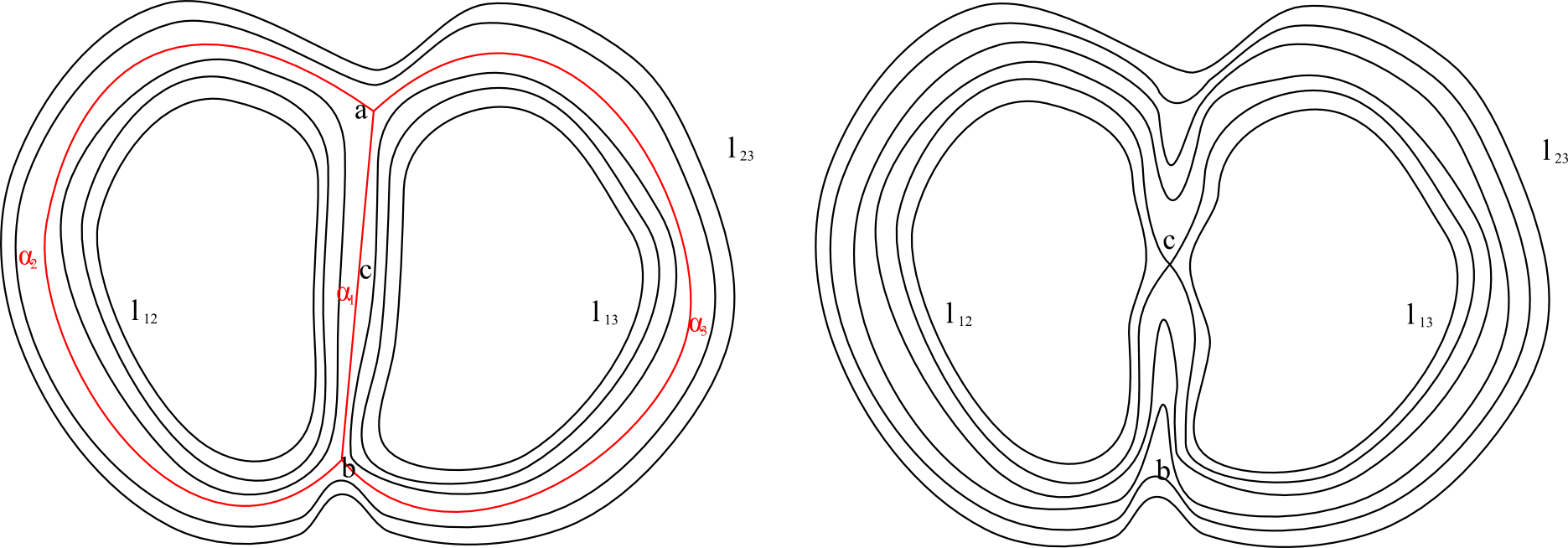}
	\caption{Surgery in the proof of Lemma \ref{theta}}. \label{fig_theta2}
\end{figure}

\begin{lemma} \label{theta}
$U(\theta)$ admits an $m-$map $f:U(\theta) \rightarrow [0,2]$
with $f(l_{23})=0$, $f(l_{12})=1$, $f(l_{13})=2$, and fibers satisfying the following inequalities:

\begin{enumerate}
\item
For $x \in [0,0.5]$, $|f^{-1}(x)|\leq |\alpha_2|+|\alpha_3|+\epsilon$

\item
For $x \in (0.5,1]$, $|f^{-1}(x)|\leq 2|\alpha_1| + |\alpha_2|+|\alpha_3|+\epsilon$

\item
For $x \in (1,2]$, $|f^{-1}(x)|\leq |\alpha_1|+|\alpha_3| +\epsilon$
\end{enumerate}

Finally, let $n_{ij}$ denote an outward unit normal at some point on $l_{ij}$,
then $df(n_{23})<0$, $df(n_{12})>0$ and $df(n_{13})>0$.

\end{lemma}

\begin{proof}
Let $a$ and $b$ denote the endpoints of $\alpha_i$.
Since $M$ is orientable, the normal tubular neighbourhood 
of $\alpha_i \cup \alpha_j$ is homeomorphic to the cylinder.
Let $l_{ij}$ be the boundary component of the tubular neighbourhood
that does not intersect $\alpha_k$, $(k \neq i,j)$. After a small perturbation in 
the neighbourhood of $a$ and $b$ we can assume that $l_{ij}$ is smooth
and there is a diffeomorphism $f_{ij}$ from $(0,1] \times S^1$ onto the region between
$\alpha_i \cup \alpha_j$ and $l_{ij}$ with $|f_{ij}(t \times S^1)| \leq |\alpha_i \cup \alpha_j| + \epsilon$. 

Let $c$ be the midpoint of $\alpha_1$.
We perform a straightforward surgery to the curves $\{f_{ij}(t \times S^1) \}$
depicted on Figure \ref{fig_theta2}. 
In the neighbourhood of $c$ we can choose a coordinate chart so that curves in the new foliation
are given by $f(x,y)=0.5 - x^2 + y^2$.
\end{proof}

Now we prove the analogue of Lemmas \ref{small area_S^2} and \ref{small area_Mk}
for $m-$functions. 

\begin{lemma} \label{m-small area disc}
For any $\epsilon > 0$ there exists $A>0$, such that for every disc $D \subset M$
with $|D| \leq A$ there exists an $m-$map $f$ from $D$ with fibers of
length less than $|\partial D| + \epsilon$. 
\end{lemma}

\begin{proof}
As in the proof of Lemma \ref{small area_S^2}, we proceed by induction on $n=\lceil \frac{|\partial D|}{\sqrt{|A|}} \rceil$.
However, the inductive assumption is now different. 
We would like to show that for every $\epsilon >0$
a subdisc $D \subset M$ admits an $m-$map with fibers of length
$$\leq |\partial D| + (4+2\sqrt{2}) \sqrt{|D|} + \epsilon'$$

Assume the Lemma to be true for all subdiscs with $\lceil \frac{|\partial D|}{\sqrt{A}} \rceil<n$.

We take small tubular neigbourhood of $\partial D$ in $D$ and foliate it by closed curves
of length $\leq |\partial D| + \epsilon'$. We subdivide the innermost curve $\gamma$
into 4 arcs of equal length.
By Besicovitch Lemma we can find an arc $\alpha$ of length $\leq \sqrt{|D|}$ connecting two opposite 
subarcs of $\gamma$. 

Let $N$ be the neighbourhood of $\gamma \cup \alpha$ as in Lemma \ref{theta}.
Note that $\partial N$ has 3 connected components:
one of them is $\partial D$ and denote the other two by $C_1$ and $C_2$, bounding
subdiscs $D_1$ and $D_2$ respectively. Assume that $D_1$ is a disc of smaller area,
hence $|D_1| < \frac{1}{2}|D|$.

$N$ admits an $m-$map $f_0:N\rightarrow [0,2]$ with fibers of length
$|f^{-1}(x)\leq |\partial D|+2\sqrt{|D|} +\epsilon'$ for $x \in [0,1]$
and $|f^{-1}(x)|\leq |\partial D_2|+\epsilon'$ for $x \in (1,2]$.

Since $|\partial D_1|\leq (n-1) \sqrt{A}$ by inductive assumption
$D_1$ admits an $m-$map $f_1$
with fibers of length $< |\partial D_1|+ \frac{4+2\sqrt{2}}{\sqrt{2}} \sqrt{|D|} +\epsilon' $.
After appropriately scaling $f_1$ on a small neigbourhood of $C_1$
and multiplying by $-1$ if necessary we can assume that $f_1(C)$ is the minimum point of $f_1$.  
Furthermore, we scale and shift $f_1$ so that $f_1(C_1)=1$ and $f_1(D_1)$ is contained between
$1$ and $1.5$
We now extend the $m-$map to $N \cup D_1$ with fibers of length 
$$\leq |\partial D_1|+ |\partial D_2| + \frac{4+2\sqrt{2}}{\sqrt{2}} \sqrt{|D|} + \epsilon'  $$
$$\leq |\partial D| + (4+2\sqrt{2}) \sqrt{|D|} +\epsilon'$$


By inductive assumption $D_2$ admits an $m-$function $f_2$ with fibers of length
$\leq |\partial D|+(4+2\sqrt{2}) \sqrt{|D|} + \epsilon'$. We modify $f_2$ so that
it takes on its minimum at $f_2(C_2)=2$. We can now extend $f$ to the whole of $D$.

Setting $A<\epsilon^2/100$ and $\epsilon' < \epsilon/10$ we obtain the desired result.

\end{proof}

\begin{lemma} \label{m-small area Mk}
For any $\epsilon > 0$ there exists $A>0$, such that for every submanifold with boundary $M_k \subset M_p$,
with $|M_k| \leq A$ there exists an $m-$map $f$ from $M_k$ with fibers of length
$\leq 2 |\partial M_k|+4(k-1) \sqrt{A}+\epsilon$.
\end{lemma}

\begin{proof}
As in the proof of Lemma \ref{small area_Mk} we connect all components of $\partial M_k$ with 
$(k-1)$ closed curves $\gamma_i$ of total length $\leq 2 (k-1) \sqrt{|M_k|}$ and denote the union 
of $\gamma_i$s and $\partial M_k$ by $C$. Denote the normal $\delta-$neighbourhood of $C$ 
in $M_k$ by $N$. After a small perturbation we can assume that the boundary of $N$
is smooth.

The complement $D$ of $N$ in $M_k$ is a disc and so by Lemma \ref{m-small area disc} it
admits an $m-$map with fibers of length $\leq |\partial M_k|+4(k-1) \sqrt{A}+\epsilon$.
Next we proceed as in the proof of Lemma \ref{small area_Mk} extending the domain of 
the $m-$function over connected components of $\partial M_k$ one by one.

Let $\{\beta_i\}_{i=1}^k$ be a collection of nested simple closed curves
in $N$ and $N_i$ denote the subset of $N$
between $\beta_i$ and $\beta_{i+1}$.
We require that $|\beta_i|\leq |\partial M_k|+4(k-1) \sqrt{A}+\epsilon$, $\beta_1 = \partial D$ and for each $i$
$N_i$ contains exactly one connected component of $\partial M_k$.
We can then apply Lemma \ref{theta} to extend the $m-$function from $D$
to $N_1$ and inductively to the whole of $M_k$.

By appropriately scaling and shifting functions obtained at each step we ensure that the desired bound
on the lengths of fibers is maintained.
\end{proof}

Now we are ready to prove Theorem \ref{morse1}.
The proof is similar to that of Theorem \ref{tree1}.
Again we proceed by induction on $\lceil \log_{4/3}(\frac{|M_p|}{A}) \rceil$.

Fix $\epsilon>0$ and let $A$ be as in Lemma \ref{m-small area Mk}.
Let $N= \lceil \log_{4/3}(\frac{|M_p|}{A}) \rceil$.
We claim that for every $M_k \subset M_p$ with $(\frac{3}{4})^{n+1}|M_p|<|M_k|\leq (\frac{3}{4})^n|M_p|$
there exists an $m-$map with fibers of length 

\begin{equation}\label{m-inductive statement}
\leq |\partial M_k| + 4 (N-n+k)\sqrt{A}+ \sum_{i=n}^{N-1} 2 c(r) (\frac{3}{4})^{i/2} \sqrt{|M_p|}
\end{equation}

Choose
$\gamma \in S(M_k,r)$ of length $\leq c(r) + \epsilon'$.

\textbf{Case 1.} $\gamma$ is a simple closed curve.

Let $N_1$ and $N_2$ denote the two components of $M_p \setminus \gamma$.
In this case there is an $m-$function $f_1$ from $N_1$ onto $[a,b]$ and
an $m-$function $f_2$ onto $[b,c]$. Hence, we obtain an $m-$function
$f:M_p \rightarrow [a,c]$ satisfying the same bound on the length of the fibers.
Using the inductive assumption we calculate that this length bound is exactly what we want.

\textbf{Case 2.} $\gamma$ is a collection of arcs with endpoints on the boundary of 
$\partial M_p$.

As in the proof Theorem \ref{tree1} we add a small collar around boundary components of
$M_p$ that intersect $\gamma$. We apply the inductive assumption to $B \subset M_p$
and extend this map to regions separated from $B$ by ``horseshoes". 
Then we extend the map to collars whose intersection with $\partial B$ is a connected arc.
We iterate this procedure until we extended the map to all of $M_p$.
By appropriately scaling the $m-$functions obtained for each region 
we ensure the correct bound on the lengths of fibers.

\end{document}